\documentclass{amsart}
\usepackage{graphicx}
\usepackage[latin2]{inputenc}
\usepackage{amsmath,amssymb,amsbsy,amsthm,amsfonts}
\usepackage{epsfig} 

\usepackage[active]{srcltx}

\begin{document}
\def\note#1{\marginpar{\small #1}}

\def\tens#1{\pmb{\mathsf{#1}}}
\def\vec#1{\boldsymbol{#1}}

\def\norm#1{\left|\!\left| #1 \right|\!\right|}
\def\fnorm#1{|\!| #1 |\!|}
\def\abs#1{\left| #1 \right|}
\def\ti{\text{I}}
\def\tii{\text{I\!I}}
\def\tiii{\text{I\!I\!I}}

\newcommand{\nm}[1]{\|#1\|}
\def\diver{\mathop{\mathrm{div}}\nolimits}
\def\grad{\mathop{\mathrm{grad}}\nolimits}
\def\Div{\mathop{\mathrm{Div}}\nolimits}
\def\Grad{\mathop{\mathrm{Grad}}\nolimits}

\def\tr{\mathop{\mathrm{tr}}\nolimits}
\def\cof{\mathop{\mathrm{cof}}\nolimits}
\def\det{\mathop{\mathrm{det}}\nolimits}
\def\spt{\mathop{\mathrm{spt}}\nolimits}
\def\dist{\mathop{\mathrm{dist}}\nolimits}

\def\lin{\mathop{\mathrm{span}}\nolimits}
\def\pr{\noindent \textbf{Proof: }}
\def\pp#1#2{\frac{\partial #1}{\partial #2}}
\def\dd#1#2{\frac{\d #1}{\d #2}}

\def\T{\mathcal{T}}
\def\R{\mathcal{R}}
\def\bx{\vec{x}}
\def\be{\vec{e}}
\def\bef{\vec{f}}
\def\bec{\vec{c}}
\def\bs{\vec{s}}
\def\ba{\vec{a}}
\def\bn{{\vec{n}}}
\def\bphi{\vec{\varphi}}
\def\btau{\vec{\tau}}
\def\bc{\vec{c}}
\def\bg{\vec{g}}

\def\bW{\tens{W}}
\def\bT{\tens{T}}
\def\bD{\tens{D}}
\def\bF{\tens{F}}
\def\bB{\tens{B}}
\def\bV{\tens{V}}
\def\bS{\tens{S}}
\def\bI{\tens{I}}
\def\bi{\vec{i}}
\def\bv{\vec{v}}
\def\bff{\vec{f}}
\def\bfi{\vec{\varphi}}
\def\bk{\vec{k}}
\def\b0{\vec{0}}
\def\bom{\vec{\omega}}
\def\bw{\vec{w}}
\def\p{\pi}
\def\bu{\vec{u}}

\def\ID{\mathcal{I}_{\bD}}
\def\IP{\mathcal{I}_{p}}
\def\Pn{(\mathcal{P})}
\def\Pe{(\mathcal{P}^{\eta})}
\def\Pee{(\mathcal{P}^{\varepsilon, \eta})}

\def\Ln#1{L^{#1}_{\bn}}

\def\Wn#1{W^{1,#1}_{\bn}}

\def\Lnd#1{L^{#1}_{\bn, \diver}}

\def\Wnd#1{W^{1,#1}_{\bn, \diver}}

\def\Wndm#1{W^{-1,#1}_{\bn, \diver}}

\def\Wnm#1{W^{-1,#1}_{\bn}}

\def\Lb#1{L^{#1}(\partial \Omega)}

\def\Lnt#1{L^{#1}_{\bn, \btau}}

\def\Wnt#1{W^{1,#1}_{\bn, \btau}}

\def\Lnd#1{L^{#1}_{\bn, \btau, \diver}}

\def\Wntd#1{W^{1,#1}_{\bn, \btau, \diver}}

\def\Wntdm#1{W^{-1,#1}_{\bn,\btau, \diver}}

\def\Wntm#1{W^{-1,#1}_{\bn, \btau}}

\newtheorem{Theorem}{Theorem}[section]
\newtheorem{Example}{Example}[section]
\newtheorem{Lem}{Lemma}[section]
\newtheorem{Rem}{Remark}[section]
\newtheorem{Def}{Definition}[section]
\newtheorem{Col}{Corollary}[section]
\newtheorem{Proposition}{Proposition}[section]

\newcommand{\Om}{\Omega}
\newcommand{ \vit}{\hbox{\bf u}}
\newcommand{ \Vit}{\hbox{\bf U}}
\newcommand{ \vitm}{\hbox{\bf w}}
\newcommand{ \ra}{\hbox{\bf r}}
\newcommand{ \vittest }{\hbox{\bf v}}
\newcommand{ \wit}{\hbox{\bf w}}
\newcommand{ \fin}{\hfill $\square$}

\newcommand{\ZZ}{\mathbb{Z}}
\newcommand{\CC}{\mathbb{C}}
\newcommand{\NN}{\mathbb{N}}
\newcommand{\V}{\zeta}
\newcommand{\RR}{\mathbb{R}}
\newcommand{\EE}{\varepsilon}
\newcommand{\Lip}{\textnormal{Lip}}
\newcommand{\XX}{X_{t,|\textnormal{D}|}}
\newcommand{\PP}{\mathfrak{p}}
\newcommand{\VV}{\bar{v}_{\nu}}
\newcommand{\QQ}{\mathbb{Q}}
\newcommand{\HH}{\ell}
\newcommand{\MM}{\mathfrak{m}}
\newcommand{\rr}{\mathcal{R}}
\newcommand{\tore}{\mathbb{T}_3}
\newcommand{\Z}{\mathbb{Z}}
\newcommand{\N}{\mathbb{N}}

\newcommand{\F}{\overline{\boldsymbol{\tau} }}

\newcommand{\nsa}{$\mathcal L(\alpha)$}
\newcommand{\ns}{$\mathcal{NS}$\ }
\newcommand{\moy} {\overline {\vit} }
\newcommand{\moys} {\overline {u} }
\newcommand{\mmoy} {\overline {\wit} }
\newcommand{\g} {\nabla }
\newcommand{\G} {\Gamma }
\newcommand{\x} {{\bf x}}
\newcommand{\E} {\varepsilon}
\newcommand{\BEQ} {\begin{equation} }
\newcommand{\EEQ} {\end{equation} }
\makeatletter
\@addtoreset{equation}{section}
\renewcommand{\theequation}{\arabic{section}.\arabic{equation}}

\newcommand{\hs}{{\rm I} \! {\rm H}_s}
\newcommand{\esp} [1] { {\bf I} \! {\bf H}_{#1} }

\newcommand{\vect}[1] { \overrightarrow #1}

\newcommand{\hsd}{{\rm I} \! {\rm H}_{s+2}}

\newcommand{\HS}{{\bf I} \! {\bf H}_s}
\newcommand{\HSD}{{\bf I} \! {\bf H}_{s+2}}

\newcommand{\hh}{{\rm I} \! {\rm H}}
\newcommand{\lp}{{\rm I} \! {\rm L}_p}
\newcommand{\leb}{{\rm I} \! {\rm L}}
\newcommand{\lprime}{{\rm I} \! {\rm L}_{p'}}
\newcommand{\ldeux}{{\rm I} \! {\rm L}_2}
\newcommand{\lun}{{\rm I} \! {\rm L}_1}
\newcommand{\linf}{{\rm I} \! {\rm L}_\infty}
\newcommand{\expk}{e^{ {\rm i} \, {\bf k} \cdot \x}}
\newcommand{\proj}{{\rm I}Ę\! {\rm P}}
\newcommand{\eu}[1]{|#1|}

\renewcommand{\theenumi}{\Roman{section}.\arabic{enumi}}

\newcounter{taskcounter}[section]

\newcommand{\bib}[1]{\refstepcounter{taskcounter} {\begin{tabular}{ l p{13,5cm}} \hskip -0,2cm [\Roman{section}.\roman{taskcounter}] & {#1}
\end{tabular}}}

\renewcommand{\thetaskcounter}{\Roman{section}.\roman{taskcounter}}

\newcounter{technique}[section]

\renewcommand{\thetechnique}{\roman{section}.\roman{technique}}

\newcommand{\tech}[1]{\refstepcounter{technique} {({\roman{section}.\roman {technique}}) {\rm  #1}}}

\newcommand{\B}{\mathcal{B}}
\newcommand{\eR}{\mathbb R}
\newcommand{\DD}{{\mathcal D}}


\title[Leray-$\alpha$ Model with Navier boundary condition]{Analysis of the Leray-$\alpha$ model with Navier slip boundary condition}

\author[H. Ali]{Hani Ali}
\address{IRMAR , UMR CNRS 6625, Universit\'{e} Rennes1, Campus Beaulieu, 35042 Rennes cedex, France}
\email{hani.ali@univ-rennes1.fr}

\author[P. Kaplick\'y]{Petr Kaplick\'y}
\address{Charles University, Faculty of Mathematics and Physics, 
Sokolovsk\'{a}~83,\\
186~75~Prague~8, Czech~Republic}
\email{kaplicky@karlin.mff.cuni.cz}


\begin{abstract}
In this paper, we establish the existence and the regularity of a unique weak solution to turbulent flows in a bounded domain $\Omega\subset\mathbb R^3$ governed by the so-called Leray-$\alpha$ model. We consider the Navier slip boundary conditions for the velocity. Furthermore, we show that, when the filter coefficient $\alpha$ tends to zero, the weak solution constructed converges to a suitable weak solution to the incompressible Navier Stokes equations subject to the Navier boundary condition. Similarly, if $\lambda\to1-$ we recover a solution to the Leray-$\alpha$ model with the homogeneous Dirichlet boundary conditions.
\end{abstract}

\thanks{Research of Petr Kaplick\'y was supported by the
grant GACR 201/09/0917, and also partially by the research project
MSM 0021620839.
}

\maketitle


\section{Introduction}
Let $\Omega \subset \mathbb{R}^3$ be a bounded domain with $C^\infty$ boundary, $T\in (0, \infty)$, and $\alpha>0$.
Our goal is to study properties of the Leray-$\alpha$ model (\nsa)   
\begin{align}
\diver \bv &=0, \label{BM}\\
\bv_{,t} + \diver (\overline{\bv} \otimes \bv) -  2\nu \diver \bD( \bv)  &= -\nabla p + \bef,\label{BLM}\\
-\alpha^2 \diver \bD( \overline{\bv}) +   \overline{\bv}+\nabla \pi
&=\bv,\quad
\diver \overline{\bv} =0.\label{TKE}
\end{align}
considered in $(0,T)\times \Omega$.
Here, all appearing quantities are smoothed. The unknown functions are the fluid velocity field $\bv$ and the pressure $p$. The external body force $\bef$ and the viscosity $\nu>0$ are given.

The system is completed by an initial condition  
\begin{equation}
\bv(0,x)=\bv_0(x) \quad \textrm{ in } \Omega,
\label{ID}
\end{equation}
and a boundary condition
\begin{align}
\bv \cdot \bn =0&, \label{bc1}\quad
\lambda \bv_{\btau} + (1-\lambda) ( \bD( \bv)
\bn)_{\btau}=0 \quad \textrm{ on } (0,T)\times \partial \Omega,
\\
\overline{\bv} \cdot \bn =0&,\quad
\lambda \overline{\bv}_{\btau} + (1-\lambda)  ( \bD( \overline{\bv})
\bn)_{\btau}=0 \quad \textrm{ on } (0,T)\times \partial \Omega.
\label{bc123}
\end{align}

Here, $\bn=\bn(\bx)$ is the outer normal located at $\bx\in \partial
\Omega$ to the boundary, $\bw_{\btau}:= \bw - (\bw \cdot \bn)\bn$ is
the projection of a vector $\bw = \bw(\bx)$  to the tangent plane of
the boundary at $\bx$, and the parameter $
\lambda\in [0,1]$ homotopically connects perfect slip boundary condition
when $\lambda = 0$ with no-slip boundary conditions when $\lambda =
1$. If $0<\lambda<1$, then \eqref{bc1} is called Navier slip
boundary conditions. In this paper we assume that $\lambda$ is any
number from $[0,1)$. 


We start our investigation showing that the problem \eqref{BM}-\eqref{bc123} possesses a unique weak solution. Since the existence and regularity theory of the problem \eqref{TKE} with boundary condition \eqref{bc123} is well known, compare Lemma~\ref{lem:st} and Corollary~\ref{cor:stokes2}, $\overline{\bv}$ can be always uniquely reconstructed from $\bv$. In this sense we understand $\overline{\bv}$ in the whole article and we concentrate only on properties of $(\bv,p)$.

\begin{Theorem}\label{thm:main1}
Let 
$\bef \in L^{2}(0,T;W^{-1,2}_{\bn})$, $\bv_0\in L^2_{\bn,\diver}$. Then there exists a unique solution $(\bv, p)$ to the system \eqref{BM}--\eqref{TKE}  such that
\begin{align}
\bv &\in \mathcal{C}_{}(0,T;L^2_{\bn,\diver}) \cap
L^2(0,T;W^{1,2}_{\bn,\diver}),\label{bv12}\\
\bv_{,t}&\in L^{2}(0,T;W^{-1,2}_{\bn}),
\label{bvt}\\
p&\in L^{2}(0,T;L^{2})
\label{psp}\\
\label{psp1} \int_\Omega p d\bx&=0\quad\mbox{for a.e. $t\in(0,T)$}.
\end{align}
and
\begin{equation}
\begin{split}
\int_0^T \langle \bv_{,t}, \bw \rangle  -  (\overline{\bv} \otimes \bv, \nabla
\bw) +
\frac{2\nu\lambda}{1-\lambda}(\bv, \bw)_{\partial \Omega} +  2\nu(
\bD(\bv), \bD(\bw) )\; dt\\
= \int_0^T (p, \diver \bw) +  \langle \bef, \bw \rangle \; dt
\qquad \textrm{ for all } \bw\in L^{2}(0,T; W^{1,2}_{\bn}).
\end{split}\label{weak1}
\end{equation}
The initial conditions are attained in the following sense
\begin{equation}
\lim_{t\to 0+}\|\bv(t)-\bv_0\|_2^2 =0. \label{inca}
\end{equation}
Moreover, the solution ($\vec{v}, p)$ satisfies the local energy equality
\begin{equation}
\label{local alpha}
\begin{array}{lcc}
\displaystyle \frac{1}{2}\int_{\Omega} (|{\vec{v}}|^2\phi)(t,\vec{x}) \ d\vec{x}+ \nu\int_{0}^{t}\int_{\Omega}|\nabla \vec{v}_{}|^{2}\phi \ d\vec{x}dt\\
\hskip 1cm  \displaystyle  =   \displaystyle   \frac{1}{2}\int_{\Omega} {|\vec{v}_0|}^2\phi(0,\vec{x}) \ d\vec{x}  +  \displaystyle \int_{0}^{t}\int_{\Omega} \frac{|\vec{v}_{}|^{2}}{2} \left(\phi_t + \nu \Delta \phi \right) \ d\vec{x}dt\\
 \hskip 2cm \displaystyle + \int_{0}^{t}\int_{\Omega}\displaystyle \left( \frac{|\vec{v}_{}|^{2}}{2} \overline{\vec{v}_{}}+ p\vec{v}_{}\right) \cdot \nabla \phi   \  d\vec{x}dt+ \displaystyle \int_{0}^{t}  \langle \bef,  \vec{v}\phi\rangle \ dt,
\end{array}
\end{equation}
for all $t \in (0,T)$ and for all  
non negative functions  $\phi \in C^{\infty}(\Omega\times\mathbb R)$ and  $\spt \phi \subset\subset \mathbb R\times\Omega$.
 \label{TH1}
\end{Theorem}
In the next theorem we focus our attention to the regularity of the unique weak solution of \eqref{BM}-\eqref{bc123}. First, we define the spaces of initial conditions. We follow \cite{Ste2006}. We set for $q\geq2$
$$
\DD_q:=
\{\varphi\in B^{2(1-\frac1q)}_{q,q}\cap L^q_{\bn,\diver}: \mbox{\eqref{bc1} holds if $q>3$}\}.
$$
Here the spaces $B^\alpha_{p,p}$ are the standard Besov spaces, see \cite[Section~2.2]{Ste2006}. Note that $\DD_2=W^{1,2}_{\bn,\diver}$.

Now we can formulate the maximal regularity result
\begin{Theorem}\label{lem:reg1}
 Assume 
$q\geq2$, $q\neq3$, $\bef \in L^{q}(0,T;L^q_{\bn,\diver})$ and ${\textbf{v}}^{}_{0} \in \DD_q$. Then the unique weak solution of the problem  \nsa\  with initial boundary condition \eqref{ID} and boundary condition \eqref{bc1}, \eqref{bc123}  is regular, i.e. ${{\bv}}^{} \in L^{q}(0,T;W^{2,q}_{\bn,\diver}),$ $ \displaystyle  {\bv}_{,t} \in L^{q}(0,T;L^q_{\bn,\diver})$ and $p^{} \in  L^{q}(0,T;W^{1,q}_{})$.
\end{Theorem}

Further we are interested in behavior of the unique weak solution to \eqref{BM}-\eqref{bc123} if 
$\alpha\to0+$, see Theorem~\ref{thm:COVVERGENCE}, $\lambda\to1-$, see Theorem~\ref{thm:llimit}, or
$\lambda\to1-$ and $\alpha\to0+$ simultaneously in Theorem~\ref{thm:last}.

Leray \cite{Leray34} was the first one who  regularized the Navier Stokes equations by smoothing  the convective velocity with regularization made by convolution.
The $\alpha$~models are based on a smoothing obtained
with the application of the inverse of the Helmholtz operator $ I - \alpha^2 \Delta.$ 
There exists a large family of the $\alpha$ models, see for example \cite{dunca06, LL03, CLT06, CHOT05,FDT02,FhT01, ILT05,A01}. 

One of the first $\alpha$ model is the Lagrangian averaged Navier Stokes equations (LANS-$\alpha$) \cite{CFHOTW99b} that was introduced as a sub-grid scale turbulence model. In \cite{FDT02} the authors suggest the LANS-$\alpha$ as a closure model for the Reynolds averaged equations.
The Leray-$\alpha$ model \cite{CHOT05}, as the other family of the $\alpha$ models, enjoys the same results of existence and uniqueness of the solutions and  was also used as a closure models for the Reynolds averaged equations.
The Leray-$\alpha$ was tested numerically in \cite{CHOT05,GH05}. In the numerical simulation the authors showed that the large scales of motion bigger than $\alpha$ in flow are captured.  It was shown also that for scales of motion smaller than $\alpha$, the energy spectra decays faster in comparison to that of Navier Stokes equations.
In \cite{CHOT05}, the convergence of a weak solution of the Leray-$\alpha$ to a weak solution of the Navier-Stokes equations as $\alpha \rightarrow 0$ was established.
It is shown in \cite{A01} that the Leray-$\alpha$ equations give rise to a suitable weak solution to the Navier-Stokes equations. All previously mentioned results were derived under the periodic boundary conditions.

 There is only a few studies on the $\alpha$ models on bounded domains. The global existence and the uniqueness of weak solutions to the LANS-$\alpha$ 
 on bounded domain with no-slip boundary condition is given in \cite{Coutand}. 
The fact that we are able to
establish such results of existence, uniqueness and convergence with Navier slip boundary conditions to the \nsa\ model  is a  novel feature of the the present
study. 
The use of the $\alpha$-equations as a model of turbulent flows in more complicated geometries 
remains to be studied.

 Finally, one may ask questions about other closure model of turbulence on  bounded domains with usual boundary conditions,
such as the Navier slip conditions.  
 This is an crucial  problem, because  
the filter in this case does not  commutes with the differential operators \cite{berselli, galdilayton, LL03, dunca06}.

The paper is organized as follows. In Sect. 2 we introduce  relevant function spaces and we recall  some preliminary results concerning solutions of elliptic equations with Navier boundary conditions. 
 Then, in Sect. 3, inspired by result in \cite{BuFeMa09}, we give the proofs of Theorems  \ref{thm:main1} and \ref{lem:reg1}. 
 In Sect. 4 we
concentrate on analysis of the behavior of the solutions $
(\bv^{\alpha}, p^{\alpha})$;  as $\alpha \to 0{+}$, where we show that the $\alpha$ regularization gives rise to a suitable weak solution to the Navier-Stokes equations.
 In Sect. 5 we  take care of the dependence of the solution of the parameter $\lambda $ in order to pass to the limit as $\lambda\to1-$ and in the last section we pass to the limit as $\alpha \to 0{+}$ and $\lambda\to1-$ simultaneously.

\section{Notation and auxiliary results}

\subsection{Notation}
We use standard notation for Lebesgue, Sobolev and Besov spaces on a domain $O$ and their norms, e.g. $L^2(O)$, $W^{1,2}(O)$, $B^{1}_{2,2}(O)$ ($=W^{1,2}(O)$ if $O$ is smooth). If $O=\Omega$ we drop $(\Omega)$, e.g. $L^{5/2}$. By $(\cdot,\cdot)_O$ we denote the inner product in $L^2(O)$, $\langle\cdot,\cdot\rangle$ stand for a duality pairing. We do not distinguish the scalar and vectorial spaces. The correct meaning is always clear from the context. Next we define relevant function spaces for the velocity field. Let $k\in\mathbb N$, $p,q\geq 1$, then
\begin{align*}
W^{k,p}_{\bn}&:=\left\{ \bv \in W^{k,p}: \; \bv \cdot \bn
=0 \textrm{ on } \partial \Omega \right\},\\
W^{k,p}_{\bn,\diver}&:=\left\{ \bv \in W^{k,p}_{\bn}: \; \diver \bv
=0 \textrm{ in }  \Omega \right\},\\
W^{-k,p'}_{\bn}&:=\left(W^{k,p}_{\bn} \right)^{*}, \quad
W^{-k,p'}_{\bn,\diver}:=\left(W^{k,p}_{\bn,\diver} \right)^{*},\\
L^q_{\bn,\diver}&:= \overline{W^{1,q}_{\bn,\diver}}^{\|\, \|_{q}}.
\end{align*}

\subsection{Stokes problem}
In this subsection we collect some known results concerning properties of solutions to Stokes problem with Navier boundary condition \eqref{bc1}.

Let us first consider the stationary Stokes problem for some fixed function $\bv$. 

\begin{align}
\label{s1} -\alpha^2 \diver \bD( \overline{\bv}) +   \overline{\bv} + \nabla \pi =\bv,\quad \diver \overline{\bv}=0\quad \textrm{ on } \Omega,\\
\label{s2}\overline{\bv} \cdot \bn =0, \quad
\lambda {\overline{\bv}}_{\btau} + (1-\lambda)(\bD( \overline{\bv})
\bn)_{\btau}=0 \quad \textrm{ on } \partial \Omega, \\
\label{s3}
\int_\Omega\pi d\bx=0.
\end{align}

We have the following lemma about existence and regularity of solutions.

\begin{Lem}\label{lem:st}
 Assume that 
$\alpha_0>0$, $\alpha\in(0,\alpha_0)$, $q>1$, $\bv \in L^{q}$. Then the unique solution  $(\overline{\bv},\pi)$  of system (\ref{s1})-(\ref{s3}) is in $W^{{2},q}_{\bn,\diver}\times W^{{1},q}$ and satisfies the estimates
$$
\|\overline{\bv} \|_{2,q} +\|{\pi} \|_{1,q}   \le C(\alpha) \|\bv\|_{q},\quad
\|\overline{\bv} \|_{q} \le C(\alpha_0)\|\bv\|_{q}.
$$
The constant $C(\alpha)>0$ may depend (and in a fact depends on $\alpha$) while $C(\alpha_0)>0$ may depend on $\alpha$ only through $\alpha_0$.

If moreover $k\in\mathbb{N}$, $k>1$ and $\bv\in W^{k,q}$. Then $(\overline{\bv},\pi)\in W^{{k+2},q}_{\bn,\diver}\times W^{{k+1},q}$ and the following estimate hold
$$
\|\overline{\bv} \|_{k+2,q} +\|{\pi} \|_{k+1,q}   \le C(\alpha)( \|\bff\|_{k,q} + \|\overline{\bv} \|_{q} +\|{\pi} \|_{q}).
$$
 \end{Lem}
\textbf{Proof.}
The first part of the lemma is proved in \cite[Theorem~1.3, (1)]{ShiShi2010}. The second part of the theorem follows from the result \cite[Theorem~10.5]{MR0162050}, since the Stokes operator satisfies the ellipticity condition \cite[Section I.1]{MR0162050} and Navier boundary condition is a complementary one, see \cite[Section~I.2]{MR0162050}.
\qed 


\begin{Col}\label{cor:stokes2}
Let $k\in\mathbb{N}\cup\{0\}$, $r\in[1,+\infty)$, $q>1$. Assume that $\bv \in L^{r}(0,T;W^{k,q})$. Then the unique solution  $(\overline{\bv},\pi)$  of problem \eqref{TKE} with boundary conditions \eqref{bc123} and \eqref{s3} satisfies  $\overline{\bv}\in L^{r}(0,T;W^{{k+2},q}_{\bn,\diver})$, $\pi\in L^{r}(0,T;W^{{k+1},q}_{\bn,\diver})$.
\end{Col}

Now we turn our attention to the evolutionary variant of the problem \eqref{s1}.
\begin{equation}
\diver \bv =0,\quad \bv_{,t} -  2\nu \diver \bD( \bv)  = -\nabla p + \bef.\label{es1}
\end{equation}
\begin{Lem}\label{lem:evst}
Let $2\leq q<+\infty$, $q\neq 3$. If $\bv_0\in \DD_q$ and $\bef\in L^q(0,T;L^q)$ then problem \eqref{es1} with \eqref{psp1}, boundary condition \eqref{bc1} and initial condition \eqref{ID} admits a unique solution $(\bv,p)$ such that 
\begin{align*}
\bv\in L^q(0,T;W^{2,q})\cap W^{1,q}(0,T;L^q),\quad
p\in  L^q(0,T;W^{1,q}).
\end{align*}
\end{Lem}
\textbf{Proof.}
This Theorem is proved in \cite[Theorem~1.2]{Shi2007}.
\qed

\subsection{Auxiliary lemma}
We  finish this section by the following interpolation lemma.
\begin{Lem}\label{lem:int1} 
Let $\Omega\subset\RR^n$ be bounded Lipschitz domain, $r>1$ and $f\in L^\infty(0,T;L^r)\cap L^r(0,T;W^{2,r})$. Then $\nabla f\in L^s(Q)$ for $s=r+r^2/(n+r)$. 
\end{Lem}
\textbf{Proof.} First we realize that the inequality  
$$
\norm{\nabla f}_s\leq C\norm{f}_{r}^{1-\theta}\norm{f}_{2,r}^\theta,
$$
with $\theta=(n+r)/(n+2r)$ holds as a consequence of \cite[4.2.1/3]{Triebel1978}, \cite[2.4.2/11 and 4.3.2/Theorem 2]{Triebel1978}, \cite[Theorem 4.6.2a]{Triebel1978}. Taking the $s$ power of this inequality the statement of the lemma then follows since $\theta s=r$.
\qed

\section{Proof of Theorems~\ref{thm:main1} and \ref{lem:reg1}}

\subsection{Proof of Theorem \ref{TH1}}
We prove the theorem using Schauder fixed point theorem. To this end we fix $r>1$, $q>1$ (the exact values of $r$ and $q$ will be determined later) and study properties of the mapping 
$$
M_2:L^2(0,T;W^{1,2}_{\bn,\diver})\cap L^r(0,T;L^{q})\to L^2(0,T;W^{1,2}_{\bn,\diver})\cap L^\infty(0,T;L^2),\quad M_2(\overline{\bv})=\bu,
$$
where $\bu\in L^2(0,T;W^{1,2}_{\bn,\diver})\cap L^\infty(0,T;L^2)$ is the unique solution of the problem
\begin{align*}
\diver \bu =0,\quad 
\bu_{,t} + \diver (\overline{\bv} \otimes \bu) -  2\nu \diver \bD( \bu)  = -\nabla p + \bef,
\end{align*}
with an initial condition
\begin{equation*}
\bu(0,x)=\bv_0(x) \quad \textrm{ in } \Omega,
\end{equation*}
and a boundary condition
\begin{align*}
&\bu \cdot \bn =0, \quad\lambda \bu_{\btau} + (1-\lambda) ( \bD( \bu)
\bn)_{\btau}=0 \quad \textrm{ on } (0,T)\times \partial \Omega.
\end{align*}
Our first goal is to determine the constants $r$, $q$ such that the mapping $M_2$ is well defined and continuous. Since for any $\gamma\geq 2$
\begin{equation}\label{embedd}
L^\infty(0,T;L^2)\cap L^2(0,T;W^{1,2})\hookrightarrow
L^\gamma(0,T;L^{\frac{6\gamma}{3\gamma-4}})
\end{equation}
it is enough to assume for some $\gamma>2$ that 
\begin{equation}\label{ass:rq}
r\geq\frac{2\gamma}{\gamma-2},\quad q\geq\frac{3\gamma}2.
\end{equation}
Under this assumptions $|\overline \bu||\bu|\in L^2(0,T;L^2)$ and the correctness of the definition of $M_2$ and its continuity follow by standard technique. Moreover, it is also seen that there exists $C>0$ independent of $\overline{\bv}$ that
\begin{equation}
\label{ass:bdd}
\nm{\bu}_{L^\gamma(0,T;L^{\frac{6\gamma}{3\gamma-4}})}+
\nm{\bu}_{L^\infty(0,T;L^2)}+
\nm{\bu}_{L^2(0,T;W^{1,2})}
\leq C.
\end{equation}
Condition \eqref{ass:rq} also assures that 
$$
\bu_t\in L^2(0,T;\big(W^{1,2}_{n,\diver}\big)^*)
$$
and Aubin-Lions compactness lemma provides that 
\begin{equation}\label{cpm2}
M_2:L^2(0,T;W^{1,2}_{\bn,\diver})\cap L^r(0,T;L^{q})\hookrightarrow L^\gamma(0,T;L^s)
\end{equation} 
is compact for any $\gamma>2$ and $s\in(1,6\gamma/(3\gamma-4))$. Compare \eqref{embedd}.

For $s\in(1,3/2)$ we introduce a mapping
$$
M_1:L^\gamma(0,T;L^s)\hookrightarrow L^\gamma(0,T;W^{2,s}), \quad M_1(\bv)=\overline{\bv},
$$
where $\overline{\bv}$ is the unique solution of the problem \eqref{s1}-\eqref{s3}. Its existence and regularity is assured by Corollary~\ref{cor:stokes2}. Here $\gamma$, $s$, $r$ and $q$ are sought such that
\begin{equation}\label{em1}
L^\gamma(0,T;W^{2,s})\hookrightarrow L^r(0,T;L^q)\cap L^2(0,T;W^{1,2}).
\end{equation}
It is needed $\gamma\geq r$, $\gamma\geq 2$ and $3s/(3-2s)\geq q$, $3s/(3-s)\geq 2$. 

Finally we want to apply Schauder fixed point theorem to $M=M_2\circ M_1$. To this end we set $\gamma=r=q=5$. In order to have $M$ well defined we need \eqref{em1} which is verified if $s>6/5$. The compactness of $M$ follows from \eqref{cpm2} provided $s<30/11$. It is seen that we can fix $s\in(6/5,3/2)$. Altogether we got that 
$$
M:L^5(0,T;L^s)\hookrightarrow L^5(0,T;L^s)
$$
is continuous, compact mapping that maps a certain ball into itself, see \eqref{ass:bdd}. Schauder fixed point theorem gives a fixed point of $M$ which solves \eqref{BM}-\eqref{bc123} in the weak sense and satisfies \eqref{bv12}, \eqref{bvt} and \eqref{inca}. It remains to reconstruct pressure. This can be done as in \cite[Section~3.2]{BuMaRa09} since in $W^{1,2}_n$ holds Helmholtz decomposition, compare \cite[Section~2.3]{BuMaRa09}. The procedure gives \eqref{psp}-\eqref{weak1}.

Up to now we proved the existence of the weak solution. Now we concentrate to its uniqueness. 
Let $({\bv_1,p_1})$ and $({\bv_2,p_2})$ be any two solutions of \nsa\ on the interval $[0,T]$, with initial values $\bv_1(0)$ and $\bv_2(0)$. Let us denote by  $\textbf{w}_{} =\bv_1-\bv_2$ and $\overline{\textbf{w}}_{} =\overline{\bv_1}-\overline{\bv_2}$.
We subtract the equation for $\bv_2$ from the equation for $\bv_1$ and test it with $\bw$. We get using successively Korn's inequality, embedding theorem and Lemma~\ref{lem:st} 
\begin{equation}
 \begin{aligned}
\nm{\bw_{,t}}_{2}^{2} +4\nu \nm{\bD(\bw_{})}_{2}^{2}  &\le
\frac{C}{\nu} \nm{\overline{\bw}\bu_{1}}_{2}^{2}+\nu(\nm{\bw}_2^2+\nm{\bD(\bw)}_2^2)\\
&\le\frac{C}{\nu} \nm{\overline{\bw}}_{2,2}^2\nm{\bu_{1}}_{1,2}^{2}+\nu(\nm{\bw}_2^2+\nm{\bD(\bw)}_2^2)\\
&\le \nm{\bw}_{2}^2(\frac{C}{\nu}\nm{\bu_{1}}_{1,2}^{2}+\nu)+\nu\nm{\bD(\bw)}_2^2.\\
\end{aligned}
\end{equation}
Using Gronwall's inequality we conclude 
the continuous dependence of the solutions on the initial data in the $L^{\infty}(0,T,L^2_{\bn,\diver})$  norm. In particular, if ${\textbf{w}}^{}_{0}=0$ then ${\textbf{w}}=0$ and the solution $\bv$ is unique.
Since the pressure part of the solution is uniquely determined by the velocity part and the condition \eqref{psp1}, the proof of the uniqueness is finished.

It remains to prove that the unique solution   $(\vec{v},p)$ verifies the local energy equality (\ref{local alpha}). To this end 
let us  take $\phi\vec{v}_{}$ as test function in (\ref{weak1}). We note that the regularity of $\overline{\vec{v}}$  ensure that all the terms are well defined. In particular the integral 
$$\displaystyle \int_{0}^{T} \int_{\Omega} \overline{\vec{v}_{}}\otimes \vec{v}_{} \cdot \nabla(\vec{v}_{} \phi)   \  d\vec{x} dt 
$$ 
is finite by using the fact that  $ \overline{\vec{v}}\otimes \vec{v}  \in L^2(0,T;{L}^{2})$ at least and $\phi\vec{v} \in 
L^2(0,T;{W}^{1,2})$.\\
An integration by parts combined with   the following identity 
\begin{equation}
\int_{\Omega} \overline{\vec{v}_{}} \otimes \vec{v}_{} \cdot \nabla(\vec{v}_{} \phi)   \  d\vec{x} = \frac{1}{2}\int_{\Omega} \overline{\vec{v}_{}} |\vec{v}_{}|^2   \cdot \nabla \phi   \  d\vec{x}
\end{equation}
yields that  for all $t \in (0,T)$ and for all  non negative functions  $\phi \in C^{\infty}$ and  $\spt \phi \subset\subset \Omega\times(0,T),$
  $(\vec{v}_{}, p_{})$ verifies  \begin{equation}
\label{local}
\begin{array}{lcc}
\displaystyle  \displaystyle   \frac{1}{2}\int_{\Omega} {|\vec{v}(t)|}^2\phi(t,\vec{x}) \ d\vec{x}   + \nu\int_{0}^{t}\int_{\Omega}|\bD (\vec{v}_{})|^{2}\phi \ d\vec{x}dt\\
\hskip 1cm   =   \displaystyle   \frac{1}{2}\int_{\Omega} {|\vec{v}_0|}^2\phi(0,\vec{x}) \ d\vec{x}  +  \displaystyle \int_{0}^{t}\int_{\Omega} \frac{|\vec{v}_{}|^{2}}{2} \phi_t \  d\vec{x}dt \\
 \hskip 2cm \displaystyle + \int_{0}^{t}\int_{\Omega}\displaystyle \left( \frac{|\vec{v}_{}|^{2}}{2} \overline{\vec{v}_{}}+ p\vec{v}_{} -  \nu [\bD(\vec{v})]\vec{v}\right) \cdot \nabla \phi   \  d\vec{x}dt+ \displaystyle \int_{0}^{t}  \langle \bef,  \vec{v}\phi\rangle \ dt. 
\end{array}
\end{equation}
Integrating by parts once more in the above equality, we obtain  (\ref{local alpha}) and  the proof of Theorem \ref{TH1} is finished.

\begin{Rem}
  Since $T>0$ was arbitrary the solution constructed in Theorem~\ref{thm:main1} may be uniquely extended for all time.
\end{Rem}
\subsection{ {Proof of Theorem~\ref{lem:reg1} }}

First we realize that by Theorem~\ref{thm:main1} we know existence of a solution $\bv$ of the problem \nsa\  such that $\bv\in \mathcal{C}_{}(0,T;L^2_{\bn,\diver}) \cap L^2(0,T;W^{1,2}_{\bn,\diver})$. By Corollary~\ref{cor:stokes2} we get that $\overline{\bv}\in L^{\infty}(0,T;W^{{2},2}_{\bn,\diver})\cap L^{2}(0,T;W^{3,2}_{\bn,\diver})$ and by embedding theorem $\overline{\bv}\in L^{\infty}(Q)$. 

We know that $\nabla\bv\in L^2(Q)$. From the regularity of $\overline{\bv}$ it follows that $\diver(\overline{\bv}\otimes \bv)=[\nabla\bv]{\overline\bv}\in L^2(Q)$. Applying Lemma~\ref{lem:evst} we get $\bv\in W^{1,2}(0,T;L^2_{\bn,\diver}) \cap L^2(0,T;W^{2,2}_{\bn,\diver})$ and by Lemma~\ref{lem:int1} $\nabla\bv\in L^{s(2)}(Q)$ with function $s(r):=r+r^2/(3+r)$. 

Let us assume $\nabla\bv\in L^{r}(Q)$ with $r\in[2,q]$, then $\diver(\overline{\bv}\otimes \bv)\in L^r(Q)$ and by Lemma~\ref{lem:evst}  $\bv\in W^{1,r}(0,T;L^r_{\bn,\diver}) \cap L^r(0,T;W^{2,r}_{\bn,\diver})$. Lemma~\ref{lem:int1} gives $\nabla\bv\in L^{s(r)}(Q)$. Since for all $r\geq 2$ it holds $s(r)>r$,  the statement of the theorem follows by iterating this procedure.
\qed

\section{Passage to the limit  as $\alpha\to0+$}

If we set $\alpha=0$ in \nsa\ we obtain the Navier Stokes system \ns
	\begin{align}
\diver \bv &=0, \\
\bv_{,t} + \diver ({\bv} \otimes \bv) -  2\nu \diver \bD( \bv)  &= -\nabla p + \bef,\\
\bv(0,x)&=\bv_0(x).
\end{align}

 Our aim here is to show that the solutions of \nsa\ from Theorem \ref{thm:main1} with $ \alpha >0$ converge to a suitable weak solution to \ns.
 The notion of a suitable weak solution of \ns was introduced by Scheffer \cite{S77}.
It is related to the notion of the weak solution, however, in addition, a local energy inequality is required (see (\ref{localinequality}) below).

First we examine a connection between $\bv$ and $\overline{\bv}$.

\begin{Lem}
\label{helmholtz}
Assume that $\vec{v}_{} \in  W^{1,2}_{\bn,\diver}$ and $\overline{\bv}$ is solution to \eqref{TKE} with boundary conditions \eqref{bc123}. Then 
\begin{equation}
\begin{aligned}
\alpha^2\nm{\bD(\bv-\overline\bv)}_{2}^2+&\frac{\alpha^2\lambda}{1-\lambda}\nm{\bv-\overline\bv}_{2,\partial\Omega}^2+
2\|\overline{\vec{v}_{}}-\vec{v}_{}\|_{2}^2\\
 &\le {\alpha^{{2}} }{} (\|\bD({\vec{v}_{}})\|_2^2 + \frac{\lambda}{1-\lambda}(\bv,\bv )_{\partial \Omega}).
\end{aligned}
\end{equation}
\end{Lem}

\begin{proof}
Testing the weak formulation of \eqref{TKE} with $\bv-\overline\bv$ yields
\begin{align*}
\alpha^{{2}}\|\bD({\vec{v}}) -\bD(\overline{\vec{v}_{}} )\|_2^2 &+ \alpha^2 \frac{\lambda}{1-\lambda}(\bv-\overline{\bv},\bv-\overline{\bv} )_{\partial \Omega}  + \|{\vec{v}_{}}-\overline{\vec{v}_{}}\|_2^2 \\
&=\alpha^2(\bD(\bv),\bD(\bv-\overline\bv))_\Omega+\frac{\lambda}{1-\lambda}(\bv,(\bv-\overline\bv))_{\partial\Omega}\\
&\le \frac12\bigg({\alpha^{{2}} }{}\|\bD({\vec{v}_{}})\|_2^2 +  {\alpha^{{2}} }\| \bD({\vec{v}_{}})-\bD(\overline{\vec{v}} )\|_2^2  \\
&+\alpha^2 \frac{\lambda}{1-\lambda}(\bv,\bv )_{\partial \Omega}+\alpha^2 \frac{\lambda}{1-\lambda}(\bv-\overline{\bv},\bv-\overline{\bv} )_{\partial \Omega}\bigg)
\end{align*}
and the result follows at once.
\end{proof}

 
 \begin{Theorem}\label{thm:COVVERGENCE}
 Let $\alpha_j\to 0+$ as $j\to+\infty$, $\bv_{0} \in  L^{2}_{\bn, \diver}$, $\bef  \in L^{2}(0,T;W^{-1,2}_{\bn})$.
 Let $\vec{v}^{\alpha_j}$ be the unique solution of \nsa\ with \eqref{ID}-\eqref{bc123} and $\alpha=\alpha_j$. 
 Then there is a subsequence of $\{\alpha_j\}$, we denote it again $\{\alpha_j\}$, $\bv\in C_{weak}(0,T; L^2_{\bn,\diver})\cap L^2(0,T;W^{1,2}_{\bn,\diver})$, $p\in L^{5/3}(\Omega\times(0,T))$ with $\bv_t\in (L^{5/2}(0,T;W^{1,5/2}_{\bn}))^*$ and $\bv(0)=\bv_0$ such that as $j\to+\infty$
\begin{align} 
 \bv^{\alpha_j}&\rightharpoonup \bv &&\textrm{weakly in } L^{2}(0,T;W^{1,2}), \label{alphazero-1}\\
 \bv^{\alpha_j}_{,t}&\rightharpoonup \bv_{,t} &&\textrm{weakly in } (L^{5/2}(0,T;W^{1,5/2}_{\bn}))^*, \label{alphazero-2}\\
 \bv^{\alpha_j}&\rightarrow \bv &&\textrm{strongly in } L^{q}(0,T;L^{q}), \textrm{ for all } 1 \le q<10/3
  \label{alphazero}\\
  p^{\alpha_j}&\rightharpoonup p &&\textrm{weakly in } L^{5/3}(0,T;L^{5/3}). \label{palphazero}
  \end{align}
Consequently, $(\vec{v},p)$ is a weak dissipative solution of \ns with Navier boundary condition \eqref{bc1} and the initial condition \eqref{ID}, i.e. 
\begin{equation}
\begin{split}
\int_0^T \langle \bv_{,t}, \bw \rangle  -  ({\bv} \otimes \bv, \nabla
\bw) +
\frac{2\nu\lambda}{1-\lambda}(\bv, \bw)_{\partial \Omega} +  2 \nu(
\bD(\bv), \bD(\bw) )\; dt\\
= \int_0^T (p, \diver \bw) +  \langle \bef, \bw \rangle \; dt
\qquad \textrm{ for all } \bw\in L^{\frac{5}{2}}(0,T; W^{1,\frac{5}{2}}_{\bn}).
\end{split}\label{nsweak1001}
\end{equation}

Moreover, the solution ($\vec{v}, p)$ satisfies the following local energy inequality
\begin{equation}
\label{localinequality}
\begin{array}{lcc}
\displaystyle \frac{1}{2}\int_{\Omega} (|{\vec{v}}|^2\phi)(t,\vec{x}) \ d\vec{x}+ \nu\int_{0}^{t}\int_{\Omega}|\nabla \vec{v}_{}|^{2}\phi \ d\vec{x}dt\\
\hskip 1cm  \displaystyle  \le   \displaystyle   \frac{1}{2}\int_{\Omega} {|\vec{v}_0}|^2\phi(0,\vec{x}) \ d\vec{x}  +  \displaystyle \int_{0}^{t}\int_{\Omega} \frac{|\vec{v}_{}|^{2}}{2} \left(\phi_t + \nu \Delta \phi \right) \\
 \hskip 2cm \displaystyle + \int_{0}^{t}\int_{\Omega}\displaystyle \left( \frac{|\vec{v}_{}|^{2}}{2} {\vec{v}_{}}+ p\vec{v}_{}\right) \cdot \nabla \phi   \  d\vec{x}dt+ \displaystyle \int_{0}^{t}  \langle \bef,  \vec{v}\phi\rangle \ dt 
\end{array}
\end{equation}
for a.e. $t \in (0,T)$ and for all  non negative function  $\phi \in C^{\infty}$ and  supp $\phi \subset\subset \Omega\times(0,T).$
\end{Theorem}
\begin{proof}[Proof of Theorem~\ref{thm:COVVERGENCE}] In order to prove Theorem \ref{thm:COVVERGENCE}, we need to find estimates that are independent of $\alpha$. In this proof constant $C>0$ is independent of $\alpha$.

First, we obtain, testing \eqref{weak1} by $\bv^\alpha$, the existence of $C>0$ that for all $\alpha$ we have
\begin{equation}\label{vao1}
\frac{2\nu\lambda}{1-\lambda}\nm{\bv^\alpha}_{L^2(0,T;L^2(\partial\Omega))}+\nm{\bv^\alpha}_{L^\infty(0,T,L^2)}+\nm{\bv^\alpha}_{L^2(0,T,W^{1,2})}\leq C.
\end{equation}
By standard interpolation we get 
\begin{equation}
\label{interpolationconvergence}
\nm{\bv^{\alpha}}_{L^{10/3}(0,T;L^{10/3})}\leq C.
\end{equation}

Lemma~\ref{lem:st} gives 
\begin{equation}\label{vao2}
\nm{\overline{\bv^{\alpha}}}_{L^{10/3}(0,T;L^{10/3})}\leq C.
\end{equation}

Since we consider Navier boundary conditions and in $W^{1,5/2}_{\bn}$ there holds Helmholtz decomposition, compare \cite[Section~2.3]{BuMaRa09},
 we can conclude from \eqref{vao1}, \eqref{interpolationconvergence} and \eqref{vao2} a uniform bound 
\begin{equation}\label{vao4}
\nm{\bv^\alpha_{,t}}_{(L^{5/2}(0,T;W^{1,5/2}_{\bn}))^*}\leq C.
\end{equation}
From \cite[Remark~3.1]{BuFeMa09} we know that for all $h\in L^\infty$ and a.e. $t\in(0,T)$ 
\begin{align*}
(p^{\alpha}(t),h)=&-(\overline{\bv}^{\alpha}(t) \otimes \bv^{\alpha}(t), \nabla^2
H) +
\frac{2\nu\lambda}{1-\lambda}(\bv^{\alpha}(t), \nabla H)_{\partial \Omega} \\
\quad &+  2\nu(
\bD(\bv^{\alpha}(t)), \nabla^2H )-\langle \bef(t), \nabla H\rangle, 
\end{align*}
holds, where $H$ is solution of $-\Delta H=h$ in $\Omega$, $\partial H/\partial \bn=0$ on $\partial\Omega$, $\int_\Omega H=0$. It is seen that integrability of the pressure follows from the integrability of $\overline{\bv} \otimes \bv$, $\bD(\bv)$, $\bef$ and $\bv$. It is standard to show from \eqref{vao1}, \eqref{interpolationconvergence} and \eqref{vao2} that 
\begin{equation}\label{vao3}
\nm{p^{\alpha}}_{L^\frac53(0,T;L^\frac53)}\leq C.
\end{equation}
                                                                        
It follows from \eqref{vao1}, \eqref{vao4} and \eqref{vao3} that we can find subsequence of $\{\alpha^j\}$ and $(\bv,p)$ such that \eqref{alphazero-1}, \eqref{alphazero-2}, \eqref{palphazero} hold and $\bv\in L^\infty(0,T;L^2)$. An another subsequence can be extracted such that
 \eqref{alphazero} holds due to \eqref{vao1} and \eqref{vao4} by Aubin-Lions lemma.

To show that $(\bv,p)$ solves \eqref{nsweak1001} and \eqref{localinequality} it is necessary to pass to the limit $\alpha^j\to 0$ as $j\to +\infty$ in \eqref{weak1} and \eqref{local alpha}. This is standard if we realize that due to 
Lemma~\ref{helmholtz}
 and \eqref{vao1}  we know that there exist $C >0$ such that 
 \begin{eqnarray}
\|\overline{\vec{v}^{\alpha}}-\vec{v}^{\alpha}\|_{L^2(0,T;L^2)}^2 \le C {\alpha^{{2}} },
\end{eqnarray}
and that this fact implies (together with \eqref{alphazero} and \eqref{vao2}) that, up to a subsequence, $\overline{\bv^{\alpha_j}}\to\bv$ in $L^q(0,T;L^q)$ for all $q \in [2,\frac{10}{3})$  as $j\to+\infty$. 

It remains to show weak continuity of $\bv$, which however follows from the fact that $\bv\in C(0,T;(W^{1,5/2}_{\bn})^*)$ by \eqref{dir2} and $\bv\in L^\infty(0,T;L^2)$.
\end{proof}

\section{Passage to the limit as $\lambda\to1-$}\label{sec:l}
Now we want to take care of dependence of the solution of the parameter $\lambda$ from \eqref{bc1} and \eqref{bc123}. We will denote this dependence by superscript $\lambda$. 

When $\lambda\to1-$ in (\ref{bc1}) we obtain the   homogeneous Dirichlet boundary condition   (i.e. the condition $\bv =0$ on $(0,T)\times\partial\Omega$). In this case the problem \nsa\  with homogeneous Dirichlet boundary condition can be obtained as a limit from \nsa\  with Navier slip boundary conditions for any $\alpha>0$ by  letting $\lambda$ in (\ref{bc1}) and (\ref{bc123}) tend
to 1-. 
 \begin{Theorem}\label{thm:llimit}
 Let $\lambda_j\to 1-$ as $j\to+\infty$, $\bv_{0} \in  L^{2}_{\bn, \diver}$, $\bef  \in L^{2}(0,T;W^{-1,2}_{\bn})$.
 Let $\bv^{\lambda_j}$ be the unique solution of \nsa\ with \eqref{ID}-\eqref{bc123} and $\lambda=\lambda_j$. 

 Then there is a subsequence of $\{\lambda_j\}$, we denote it again $\{\lambda_j\}$, $\bv\in C_{}(0,T; L^2_{\bn,\diver})\cap L^2(0,T;W^{1,2}_{0,\diver})$ with $\bv_t\in (L^{2}(0,T;W^{1,2}_{0,\diver})^*$ and $\bv(0)=\bv_0$ such that as $j\to+\infty$
\begin{align} 
 \bv^{\lambda_j}&\rightharpoonup \bv &&\textrm{weakly in } L^{2}(0,T;W^{1,2}), \label{dir1}\\
 \bv^{\lambda_j}_{,t}&\rightharpoonup\bv_{,t} &&\textrm{weakly in $(L^{2}(0,T;W^{1,2}_{0,\diver}))^*$}, \label{dir2}\\
 \bv^{\lambda_j}&\rightarrow \bv &&\textrm{strongly in } L^{q}(0,T;L^{q}), \textrm{ for all } 1 \le q<10/3,
  \label{dir3}
  \end{align}
$\vec{v}$ is the unique weak solution to \nsa\ with homogeneous Dirichlet boundary condition and initial condition \eqref{ID}, i.e.
\begin{equation}
\begin{split}
\int_0^T \langle \bv_{,t}, \bw \rangle  -  (\overline{\bv} \otimes \bv, \nabla\bw) +2\nu(\bD(\bv), \bD(\bw) )\; dt
= \int_0^T \langle \bef, \bw \rangle \; dt
\end{split}\label{weak1dir}
\end{equation}
for all $\bw\in L^{2}(0,T; W^{1,2}_{0,\diver})$.

Let moreover $\bef\in L^q(0,T; L^q_{\bn,\diver})$ for some $q\geq 2$, $\bv_0\in W^{2-2/q,q}$ with $\bv_0=0$ on $\partial\Omega$ and $\diver\bv_0=0$ is $\Omega$. Then 
\begin{equation}\label{reg:strong-dir}
\bv\in L^q(0,T;W^{2,q}_{0,\diver})\cap W^{1,q}(0,T; L^q_{\bn,\diver})
\end{equation}
and a pressure can be reconstructed in such a way that $p\in L^q(0,T;W^{1,q})$ and \eqref{psp1} holds.
\end{Theorem}
\begin{proof}
Testing \eqref{weak1} with $\bv^\lambda$ we know that
\begin{equation}\begin{split}
\sup_{t\in(0,T)}\|{\bv^\lambda}(t)\|_{2}^2+ \nu\int_0^T \|{\bv^\lambda}(t)\|_{1,2}^2  dt + \nu \frac{\lambda}{1-\lambda} \int_0^T  (\bv^\lambda, \bv^\lambda)_{\partial \Omega} \le C(\bv_0,\bef) < \infty. \label{regularity NS lambda}
\end{split}
\end{equation}
Testing \eqref{TKE} by $\bv^\lambda$ we get using \eqref{regularity NS lambda} the estimate
\begin{equation}\label{regularity NS lambda-fil}
\nm{\overline\bv^\lambda}_{L^\infty(0,T;W^{1,2})}+\nm{\overline\bv^\lambda}_{L^\infty(0,T;L^6)}\leq C(\bv_0,\bef).
\end{equation}

From \eqref{regularity NS lambda} and \eqref{regularity NS lambda-fil} we get that 
$$
\nm{\overline\bv^\lambda\bv^\lambda}_{L^{5/2}(Q)}\leq C(\bv_0,\bef),
$$
and consequently
\begin{equation}\label{ns-td}
\nm{\bv^\lambda_{,t}}_{(L^{2}(0,T;W^{1,2}_{0,\diver}))^*}\leq C(\bv_0,\bef).
\end{equation}

Using \eqref{regularity NS lambda} and \eqref{ns-td} it is standard to find a subsequence $\{\lambda_j\}$ and $\bv$ such that \eqref{dir1}-\eqref{dir3} and \eqref{weak1dir} hold. The equation \eqref{weak1dir} is obtained letting $\lambda^j\to 1-$ in \eqref{weak1}. The boundary terms disappear since the test functions vanish on the boundary and the term with pressure is not present because the test functions are divergencefree.

Now we show that the trace of $\bv$ is zero. It follows from \eqref{regularity NS lambda} since 
$$
\int_0^T\nm{\bv^\lambda}_{2,\partial\Omega}^2\leq C\frac{1-\lambda}\lambda\to0\quad\mbox{as $\lambda\to 0+.$}
$$


Last, we need that $\bv(0)=\bv_0$. That follows from the initial condition for $\bv^{\lambda_j}(0)=\bv(0)$ since $\bv,\bv^{\lambda_j}\in C_{weak}(0,T;L^2_{\bn,\diver})$. (The last statement follows from the fact that $\bv, \bv^{\lambda_j}\in
C(0,T;(W^{1,5/2}_{\bn,\diver})^*)\cap L^\infty(0,T;L^2)\hookrightarrow C_{weak}(0,T;L^2_{\bn,\diver})$).

In the situation when $\bef\in (L^{2}(0,T; W^{1,2}_{0,\diver}))^*$ only it is not known how to construct pressure as a function $p\in L^{2}((0,T)\times\Omega)$, compare \cite[Section~IV.2.6]{So2001}. 
A different situation occurs if $\bef\in L^q(Q)$, $q\geq2$. Then the regularity \eqref{reg:strong-dir} of the solution $\bv$ can be shown as in Theorem~\ref{lem:reg1} since Lemmas~\ref{lem:st} and \ref{lem:evst} hold also under homogeneous Dirichlet boundary conditions, compare \cite{AG}, \cite{GiSo1991-jfa}. Having \eqref{reg:strong-dir} the pressure can be reconstructed on a.e. time level by de Rham's theorem and its regularity can be read from the equation.
\end{proof}
\section{Passage to the limit as $\lambda\to1-$ and $\alpha  \to 0+$ }
When $\lambda\to1-$ and $\alpha  \to 0+$ 
a theorem similar to Theorem~\ref{thm:llimit} 
 can be proved.
 \begin{Theorem}\label{thm:last}
Let $\lambda_j\to1-$, $\alpha_j\to 0+$, $\bv_{0} \in  L^{2}_{\bn, \diver}$, $\bef  \in L^{2}(0,T;W^{-1,2}_{\bn})$.
 Let $\vec{v}^{\lambda_j,\alpha_j}$ be the unique solution of \nsa\ with \eqref{ID}-\eqref{bc123}, $ \lambda_j=\lambda $ and   $\alpha=\alpha_j$. 
 Then there is a subsequence of $\{\lambda_j, \alpha_j\}$, we denote it again $\{\lambda_j, \alpha_j\}$, $\bv\in C_{weak}(0,T; L^2_{\bn,\diver})\cap L^2(0,T;W^{1,2}_{0,\diver})$, with $\bv_t\in (L^{2}(0,T;W^{1,3}_{0,\diver}))^*$ and $\bv(0)=\bv_0$ such that as $j\to+\infty$
\begin{align} 
 \bv^{ \lambda_j, \alpha_j}&\rightharpoonup \bv &&\textrm{weakly in } L^{2}(0,T;W^{1,2}), \label{lambdaalphazero-1}\\
 \bv^{\lambda_j,\alpha_j}_{,t}&\rightharpoonup \bv_{,t} &&\textrm{weakly in } (L^{2}(0,T;W^{1,3}_{0,\diver}))^*, \label{lambdaalphazero-2}\\
 \bv^{\lambda_j, \alpha_j}&\rightarrow \bv &&\textrm{strongly in } L^{q}(0,T;L^{q}), \textrm{ for all } 1 \le q<10/3
  \label{lambdaalphazero}
  \end{align}
Consequently,  the velocity part $\vec{v}$  is a weak dissipative solution of the Navier Stokes equations  with homogeneous Dirichlet boundary condition   and the initial condition $\bv_0$, i.e. 
\begin{equation}
\begin{split}
\int_0^T \langle \bv_{,t}, \bw \rangle  -  ({\bv} \otimes \bv, \nabla
\bw)  +  2 \nu(
\bD(\bv), \bD(\bw) )\; dt
= \int_0^T    \langle \bef, \bw \rangle \; dt\\
\qquad \textrm{ for all } \bw\in L^{2}(0,T; W^{1,3}_{0,\diver}).
\end{split}\label{lmitnsweak1001}
\end{equation}
\end{Theorem}

\begin{proof}
The proof of this Theorem follows the lines of the proof of  Theorem  \ref{thm:COVVERGENCE} and Theorem  \ref{thm:llimit}. 
First we obtain uniform estimates \eqref{vao1} and \eqref{interpolationconvergence}. Now we need to reconstruct a uniform estimates for $\overline{\bv}^{\lambda_j,\alpha_j}$. Since in Lemma~\ref{lem:st} the dependence of constants on $\lambda$ is not addressed we cannot use it. Instead we test \eqref{TKE} with $\overline{\bv}^{\lambda_j,\alpha_j}$ and get a uniform estimate 
\begin{equation}\label{vao30}
\nm{\overline{\bv}^{\lambda_j,\alpha_j}}_{L^\infty(0,T;L^2)}< C.
\end{equation}
It follows 
$$
\nm{|\overline{\bv}^{\lambda_j,\alpha_j}||{\bv}^{\lambda_j,\alpha_j}|}_{L^2(0,T;L^{\frac32})}<C\quad\mbox{and}\quad
\nm{{\bv}^{\lambda_j,\alpha_j}_{,t}}_{(L^2(0,T;W^{1,3}_{0,\diver}))^*}<C.
$$
Consequently we can extract a subsequence $(\lambda_j,\alpha_j)$ that \eqref{lambdaalphazero-1}, \eqref{lambdaalphazero-2} and by Aubin-Lions lemma also \eqref{lambdaalphazero} hold. Combining Lemma~\ref{helmholtz} with the estimate \eqref{vao1} we get that $\overline{\bv}^{\lambda_j,\alpha_j}\to \bv$ in $L^2(Q)$ and by \eqref{vao30} also in $L^s(0,T;L^2)$ for all $s>2$ as $j\to+\infty$. The limit function $\bv$ must be traceless due to \eqref{vao1}. With this information it is standard to pass to the limit as $j\to+\infty$ in \eqref{weak1} to get \eqref{lmitnsweak1001}.
\end{proof}
  
\begin{Rem}
Generally with  homogeneous Dirichlet boundary condition the existence of the pressure term $p$ of the Navier Stokes equations   is not obvious and the pressure may not exist, compare \cite{27}. 
\end{Rem}
 \begin{Rem}
 Finally we would like to notice that the results reported here can be extended to the Navier-Stokes-Voigt
 equations (see in \cite{CLT06} and the references inside)  with Navier slip boundary condition. 
 \end{Rem}


\begin{thebibliography}{10}

\bibitem{MR0162050}
S.~Agmon, A.~Douglis, and L.~Nirenberg.
\newblock Estimates near the boundary for solutions of elliptic
  partialdifferential equations satisfying general boundary conditions.{II}.
\newblock {\em Comm. Pure Appl. Math.}, 17:35--92, 1964.

\bibitem{A01}
H.~Ali.
\newblock On a critical leray-$\alpha$ model of turbulence.
\newblock {\em http://arxiv.org/abs/1103.0798}.

\bibitem{AG}
Ch{\'e}rif Amrouche and Vivette Girault.
\newblock Decomposition of vector spaces and application to the {S}tokes
  problem in arbitrary dimension.
\newblock {\em Czechoslovak Math. J.}, 44(119)(1):109--140, 1994.

\bibitem{berselli}
L.C. Berselli, T.~Iliescu, and W.J. Layton.
\newblock {\em Mathematics of Large Eddy Simulation of Turbulent Flows}.
\newblock Springer-Verlag, Berlin, 2006.

\bibitem{BuFeMa09}
M.~Bul\'{i}\v{c}ek, E.~Feireisl, and J.~M\'{a}lek.
\newblock Navier-{S}tokes-{F}ourier system for incompressible fluids with
  temperature dependent material coefficients.
\newblock {\em Nonlinear Analysis: Real World Applications}, 10(2), April 2009.

\bibitem{BuMaRa09}
M.~Bul\'{i}\v{c}ek, J.~M\'{a}lek, and K.~R. Rajagopal.
\newblock Mathematical analysis of unsteady flows of fluids with pressure,
  shear-rate and temperature dependent material moduli, that slip at solid
  boundaries.
\newblock 2009.

\bibitem{CLT06}
Y.~Cao, E.~M. Lunasin, and E.~S. Titi.
\newblock Globall well posdness of the three dimensional viscous and inviscid
  simplified {B}ardina turbulence models.
\newblock {\em Commun. Math. Sci.}, 4:823--848, 2006.

\bibitem{CFHOTW99b}
S.~Chen, C.~Foias, D.~Holm, E.~Olson, E.~S. Titi, and S.~Wynne.
\newblock The {C}amassa-{H}olm equations and turbulence.
\newblock {\em Physica D}, D133:49--65, 1999.

\bibitem{CHOT05}
A.~{C}heskidov, D.~D. Holm, E.~Olson, and E.~S. Titi.
\newblock On a {L}eray-$\alpha$ model of turbulence.
\newblock {\em Royal Society London, Proceedings, Series A, Mathematical,
  Physical and Engineering Sciences}, 461:629--649, 2005.

\bibitem{Coutand}
D.~Coutand, J.~Peirce, and S.~Shkoller.
\newblock Global well-posedness of weak solutions for the lagrangian averaged
  navier-stokes equations on bounded domains.
\newblock {\em Communications on pure and applied analysis}, 1(1):35--50, 2002.

\bibitem{dunca06}
A.~{Dunca} and Y.~{Epshteyn}.
\newblock On the {S}tolz-{A}dams deconvolution model for the large-eddy
  simulation of turbulent flows.
\newblock {\em SIAM J. Math. Anal.}, 37(6):1890--1902, 2006.

\bibitem{FDT02}
C.~Foias, D.~Holm, and E.~S. Titi.
\newblock The three dimensional viscous {C}amassa-{H}olm equations and their
  relation to the {N}avier-{S}tokes equations and turbulence theory.
\newblock {\em Journal of Dynamics and Differential Equations}, 14:1--35, 2002.

\bibitem{FhT01}
C.~Foias, D.~D. Holm, and E.~S. Titi.
\newblock The {N}avier-{S}tokes-alpha model of fluid turbulence.
\newblock {\em Physica D}, 152:505--519, 2001.

\bibitem{galdilayton}
G.~P. {Galdi} and W.~J. {Layton}.
\newblock Approximation of the larger eddies in fluid motions. ii. a model for
  space-filtered flow.
\newblock {\em Math. Models Methods Appl. Sci.}, 10(3):343--350, 2000.

\bibitem{GH05}
B.~J. Geurts and D.~D. Holm.
\newblock Leray and {L}{A}{N}{S}-alpha modeling of turbulent mixing.
\newblock {\em {J}ournal of {T}urbulence}, 00:1--42, 2005.

\bibitem{GiSo1991-jfa}
Yoshikazu Giga and Hermann Sohr.
\newblock Abstract {$L^p$} estimates for the {C}auchy problem with applications
  to the {N}avier-{S}tokes equations in exterior domains.
\newblock {\em J. Funct. Anal.}, 102(1):72--94, 1991.

\bibitem{ILT05}
A.~A. {I}lyin, E.~M. Lunasin, and E.~S. Titi.
\newblock A modified {L}eray-alpha subgrid-scale model of turbulence.
\newblock {\em Nonlinearity}, 19:879--897, 2006.

\bibitem{LL03}
W.~{L}ayton and {R. Lewandowski}.
\newblock A simple and stable scale similarity model for large eddy simulation:
  energy balance and existence of weak solutions.
\newblock {\em Applied Math. letters}, 16:1205--1209, 2003.

\bibitem{Leray34}
J.~Leray.
\newblock Sur le mouvement d'un liquide visquex emplissant l'espace.
\newblock {\em Acta Math.}, 63:193--248, 1934.

\bibitem{S77}
Vladimir Scheffer.
\newblock Hausdorff measure and the navier-stokes equations.
\newblock {\em Communication in Mathematical Physics}, 55(2):97--112, 1977.

\bibitem{ShiShi2010}
Y.~Shibata and R.~Shimada.
\newblock On a generalized resolvent estimate for the {S}tokes system with
  {R}obin boundary condition.
\newblock {\em J. Math. Soc. Japan}, 59(2):469--519, 2007.

\bibitem{Shi2007}
Rieko Shimada.
\newblock On the {$L_p$}-{$L_q$} maximal regularity for {S}tokes equations with
  {R}obin boundary condition in a bounded domain.
\newblock {\em Math. Methods Appl. Sci.}, 30(3):257--289, 2007.

\bibitem{27}
J.~Simon.
\newblock On the existence of the pressure for the solutions of the variational
  navier-stokes equations.
\newblock {\em J. Math. Fluid Mech.}, 1(3):225--234, 1999.

\bibitem{So2001}
Hermann Sohr.
\newblock {\em The {N}avier-{S}tokes equations}.
\newblock Birkh\"auser Advanced Texts: Basler Lehrb\"ucher. [Birkh\"auser
  Advanced Texts: Basel Textbooks]. Birkh\"auser Verlag, Basel, 2001.
\newblock An elementary functional analytic approach.

\bibitem{Ste2006}
Olivier Steiger.
\newblock Navier-{S}tokes equations with first order boundary conditions.
\newblock {\em J. Math. Fluid Mech.}, 8(4):456--481, 2006.

\bibitem{Triebel1978}
H.~Triebel.
\newblock {\em Interpolation theory, function spaces, differential operators}.
\newblock VEB Deutscher Verlag der Wissenschaften, Berlin, 1978.

\end{thebibliography}
\end{document}